\documentclass[twocolumn]{autart}    % Enable this line and disable the
\usepackage{graphicx}          % Include this line if your
\usepackage{color}
\usepackage{amsfonts}
\usepackage{amssymb}
\usepackage{cite}
\usepackage{amsmath}
\usepackage{algorithm}
\usepackage{algpseudocode}
\usepackage{graphics}
\usepackage{epsfig}
\usepackage{graphicx}          % Include this line if your
\usepackage{setspace}
\usepackage{epstopdf}
\usepackage{subfig}                               % document contains figures,
\newcommand{\be}{\begin{equation}}
\newcommand{\ee}{\end{equation}}
\newcommand{\bea}{\begin{eqnarray}}
\newcommand{\eea}{\end{eqnarray}}

\newtheorem{remark}{Remark}%[section]
\newtheorem{theorem}{Theorem}%[section]
\newtheorem{lemma}{Lemma}%[section]

\def\ba{\begin{array}}
\def\ea{\end{array}}
\def\bea{\begin{align}}
\def\eea{\end{align}}

\begin{document}

\begin{frontmatter}

\title{Distributed Q-Learning for Stochastic LQ Control with Unknown Uncertainty\thanksref{footnoteinfo}} % Title, preferably not more

\thanks[footnoteinfo]{This work is financially supported in part by Research Grants Council of Hong Kong under grant 15215319 and 15216720, in part by the National Natural Science Foundation of China under Grant 61922051, Grant 61633014, Grant 61873332, Grant U1806204,  Grant U1701264, and in part by the Science and Technology Project of Qingdao West Coast New Area under Grant 2019-32, 2020-20, 2020-1-4, and in part by the High-level Talent Team Project of Qingdao West Coast New Area under Grant RCTD-JC-2019-05, and in part by the Key Research and Development Program of Shandong Province under Grant 2020CXGC01208.}
\thanks{Part of this work of the first author was carried out  at College of Electrical Engineering and Automation, Shandong University of Science and Technology, Qingdao 266000, China.}

\author[Xu]{Zhaorong Zhang}\ead{zhaorong.zhang@uon.edu.au},
\author[Xu]{Juanjuan Xu}\ead{juanjuanxu@sdu.edu.cn}   % Add the
 and             % e-mail address
\author[Li]{Xun Li}\ead{li.xun@polyu.edu.hk},  % (ead) as shown

\address[Xu]{School of Control Science and Engineering, Shandong University, Jinan 250061, China.}  % Please supply
\address[Li]{Department of Applied Mathematics, The Hong Kong Polytechnic University, Kowloon, Hong Kong 999077, China}
         % here.

                             % keyword list or with the
                                          % help of the Automatica
                                          % keyword wizard

\begin{abstract}                          % Abstract of not more than 200 words.

This paper studies a discrete-time stochastic control problem with linear quadratic criteria over an infinite-time horizon. We focus on a class of control systems whose system matrices are associated with random parameters involving unknown statistical properties. In particular, we design a distributed Q-learning algorithm to tackle the Riccati equation and derive the optimal controller stabilizing the system. The key technique is that we convert the problem of solving the Riccati equation into deriving the zero point of a matrix equation and devise a distributed stochastic approximation method to compute the estimates of the zero point. The convergence analysis proves that the distributed Q-learning algorithm converges to the correct value eventually. A numerical example sheds light on that the distributed Q-learning algorithm converges asymptotically. 
\end{abstract}

\end{frontmatter}

\section{Introduction}
A large number of phenomena arising from chemistry, biology, ecology, economics, physics, engineering and other disciplines can be modeled as stochastic systems in nature. The past few decades have witnessed the rapid development of the studies of stochastic systems, initiated by Wonham\cite{w} in the 1960s. In particular, stochastic linear-quadratic (LQ) control has become a research focus widely used in engineering systems. For stochastic systems with additive noises, Elia \cite{ei} solved the mean-square stabilization problem of an LQ system with fading channels involving random parameters and derived the optimal feedback solution of a generalized Riccati equation using robust control. Taksar-Poznyak-Iparraguirre \cite{MI} focused on the output feedback regulation problems where the output signal is affected by additive noises, and they devised a robust linear control strategy. Although there have been great success in the research of stochastic systems with additive noises, the systems with multiplicative noises can better model the uncertainties of the system, such as packet dropouts, quantization errors, the constraints on signal-to-noise ratios and bandwidth limits. Therefore, stochastic LQ control with multiplicative noises has been widely studied in diverse disciplines. For example, Willems-Blankenship \cite{fr} derived the necessary and sufficient conditions for the stability of a class of linear systems whose open loop involves white noise multiplicative gain. Zhang-Li-Xu-Fu \cite{zxl} considered the problem of LQ regulation with multiplicative noise and time-delay, and proposed an optimal controller for stabilizing a class of discrete-time stochastic systems based on the solution of the Riccati-ZXL difference equation.

However, in the aforementioned literature, system matrices and parameters should be known in advance, which is unrealistic in most applications. When faced with an unknown system model, the above methods will lose efficiency. In the past few decades, a lot of efforts have been continuously devoted to stochastic control problems with incomplete system dynamics and parameters. In particular, reinforcement learning (RL) algorithms have been widely used in control problems. In a variety of machine learning and computational intelligence problems, the RL method is regarded as a general technology that drives the agent to interact with its environment and obtain the best control policy based on the rewards received from the environment. Generally, the core procedure of RL is policy iteration and value iteration, which enable the RL method to solve control problems with unknown parameters and incomplete system dynamics. Zhang-Cui-Luo-Jiang \cite{he} designed a data-driven RL method that can solve the $H_\infty$ control with unknown nonlinear dynamics and constrained control inputs. Zhao-Liu-Luo \cite{bo} integrated RL and feedforward neural network to develop an online algorithm to solve the stability problem of optimal control with uncertain input constraints. Wang-Zhang-Luo \cite{tao} applied an adaptive dynamic programming algorithm based on value iteration to tackle stochastic LQ control with unknown system parameters in infinite time. Especially, the Q-learning algorithm is a kind of RL methods widely used in LQ control and regulation problems. Li-Chai-Lewis-Ding-Jiang \cite{af} proposed an off-policy Q-learning algorithm, which provides an approximate optimal control strategy for affine nonlinear systems with unknown system dynamics. Xu-Jagnannathan \cite{hao} employed the techniques of adaptive estimator and Q-learning to generate a direct dynamic programming algorithm, which solves the problem of stochastic optimal adjustment. However, in the aforementioned literature, obtaining the solution to algebraic Riccati equation (ARE) and the optimal controller must depend on the knowledge of the statistical information of random parameters, which is not feasible in practical applications\cite{zh}. Recently, Du-Meng-Zhang \cite{du} devised a Q-learning algorithm to solve the Bellman equation of discrete-time LQ control problems, where system parameters are associated with multiplicative noise with unknown Gaussian distribution. 

Note that the algorithm proposed by Du-Meng-Zhang \cite{du} is centralized, which is risky to a certain extent, because all data is transmitted in one packet, which creates an opportunity for an attacker to steal all information. To this end, in this paper, we propose a distributed Q-learning algorithm for stochastic control problems with random parameters involving unknown statistical information. Inspired by the work in \cite{sa}, we design a novel distributed stochastic approximation algorithm. Under the distributed stochastic approximation scheme, we can approximate the zero point of a matrix equation, parameterize it to derive the solution of the Riccati equation, and design the optimal controller. The convergence of the distributed stochastic approximation algorithm has been presented, and its correctness has been verified by numerical examples. 

The rest of this paper is arranged as followed. Problem formulation and preliminaries are introduced in Section \uppercase\expandafter{\romannumeral2}. The main results for the proposed distributed algorithm and its convergence analysis are presented in Section \uppercase\expandafter{\romannumeral3}. Section \uppercase\expandafter{\romannumeral4} displays the simulation results of running the proposed algorithm. Section \uppercase\expandafter{\romannumeral5} provides the conclusion of this paper.

Notations: Define $\mathbb{R}^n$ as the set of real $n\times n$-matrices. Denote $\mathbb{N}$ as the set of integer numbers. Let $\mathbb{S}_+^n$ be the set of positive semidefinite $n\times n$-matrices. $\Vert.\Vert$ stands for the standard $\ell$-2 norm on vectors or their induced norms on matrices. $I_n$ represents the identity matrix with dimension $n\times n$. Given a matrix $A$, denote its pseudo inverse matrix as $A^{\dag}$. Let $A^T$ be the transpose of matrix $A$. The Kronecker product is denoted by $\otimes$. $A\textgreater 0$ indicates that $A$ is a positive-definite matrix. Denote $\mathbb{E}[A]$ as the expectation of matrix $A$. 
\section{Problem Formulation and Preliminaries}
In this section, we establish the problem of LQ control and introduce existing methods for solving this problem. 
\subsection{Problem Formulation}
We consider a discrete-time system described by:
\begin{align}\label{w1}
	x(k+1)=A(k)x(k)+B(k)u(k),
\end{align}
where $x(k)\in \mathbb{R}^n$ is the state, $u(k)\in \mathbb{R}^m$ is the control input, $A(k)$ and $B(k)$ are random matrices with compatible dimensions, which are written as:
\begin{align*}
	A(k)=A+\bar{A}\omega(k),~
	B(k)=B+\bar{B}\omega(k).\nonumber
\end{align*}
Here, $A$, $\bar{A}$, $B$ and $\bar{B}$ are constant matrices, $\omega(k)$ is a multiplicative noise which follows the Gaussian distribution $N(\mu,\sigma^2)$. In particular, the statistics information $\mu$ and $\sigma^2$ are unknown when designing the controller. 

The cost functional is defined as
\begin{align}
	J(x,u)=\sum_{k=0}^{\infty}\left[\begin{matrix}
		x(k)^T&u(k)^T
	\end{matrix}\right]N\left[\begin{matrix}
		x(k)\\u(k)
	\end{matrix}\right],
\end{align}
where $N$ is a diagonal matrix denoted as
\begin{align}
	N=\left[\begin{matrix}
		Q&0\\
		0&R
	\end{matrix}\right],\nonumber
\end{align}
where $Q\textgreater 0$ and $R\textgreater 0$ are constant matrices. Since the cost functional is given over an infinite horizon, the controller is chosen from the stabilizing ones. Namely, the admissible controller set is defined as $U=\{u(k)\Big|\mathbb{E}\sum\limits_{k=0}^{\infty}\Vert u(k)\Vert^2\textless 0, k\in\mathbb{N}\}$. 

Our objective is to minimize the expected value of the cost functional and obtain the optimal and stabilizing controller in admissible control set $U$.  
\begin{rem}
	We emphasize that the LQ control studied in this paper is associated with multiplicative noise involving unknown statistical properties.  Although the research on the LQ control problem of known system parameters has been quite mature, how to solve LQ control with unknown uncertainty in a distributed manner is still an open question. 
\end{rem}
\subsection{Preliminaries}
The stabilizability of system (\ref{w1}) is essentially equivalent to determine whether the value function
\begin{align*}
	V(k,x(k)):=\min\limits_{u(s),s\geq k}\sum\limits_{s=k}^{\infty}\mathbb{E}[x(s)^TQx(s)+u(s)^TRu(s)]
\end{align*}
is finite for all $x\in \mathbb{R}^n$. In addition, in the case where the system matrices are known, the generalized ARE can be expressed as:
\begin{eqnarray}\label{q1}
	P&\hspace{-2mm}=\mathbb{E}[Q+A(k)^TPA(k)]-\mathbb{E}[A(k)^TPB(k)]\nonumber\\
	&\times\mathbb{E}[B(k)^TPB(k)+R]^{-1}\mathbb{E}[B(k)^TPA(k)]
\end{eqnarray}
We have the following lemma: 
\begin{lemma}
	Assume that equation (\ref{q1}) has a unique solution $P\textgreater 0$, then the optimal controller is given by 
	\begin{align}\label{n1}
		u^*(k)=&-\{\mathbb{E}[B(k)^TPB(k)+R]\}^{-1}\mathbb{E}[B(k)^TPA(k)]x(k).\nonumber\\
	&\quad	
	\end{align} 
	Also, the value function has the form of 
	\begin{align}
		V(k,x(k))=\mathbb{E}[x(k)^TPx(k)].
	\end{align}	
\end{lemma}
\begin{proof}	
	Define
	\begin{align}\label{e7}
		V(s,x(s))=\mathbb{E}[x(s)^TPx(s)],
	\end{align} 
	where $P$ is the positive-definite solution to equation (\ref{q1}). It follows from (\ref{e7}) that we have
		\begin{align}\label{p1}
			&\mathbb{E}[V(s+1,x(s+1))-V(s,x(s))]\nonumber\\
			&=\mathbb{E}[x^T(s+1)Px(s+1)-x(s)^TPx(s)]\nonumber\\
			&=\mathbb{E}\{[A(s)x(s)+B(s)u(s)]^TP[A(s)x(s)+B(s)u(s)]\nonumber\\
			&\hspace{4mm}-x^T(s)Px(s)\}\nonumber\\
			&=\mathbb{E}[x^T(s)A^T(s)PA(s)x(s)+u^T(s)B^T(s)PA(s)x(s)\nonumber\\
			&\hspace{4mm}+x^T(s)A^T(s)PB(s)u(s)+u^T(s)B^T(s)PB(s)u(s)\nonumber\nonumber\\
			&\hspace{4mm}-x^T(s)Px(s)]\nonumber\\
			&=\mathbb{E}\{x(s)^T[A(s)^TPA(s)-P-Q]x(s)\nonumber\\
			&\hspace{4mm}+u(s)^TB(s)^TPA(s)x(s)+x(s)^TA(s)^TPB(s)u(s)\nonumber\\
			&\hspace{4mm}+u(s)^T(B(s)^TPB(s)+R)u(s)\nonumber\\
			&\hspace{4mm}-x(s)^TQx(s)-u(s)^TRu(s)\}.
	\end{align}
	Calculate the summation of equation (\ref{p1}) for $s$ from $k$ to $\infty$:
		\begin{align}\label{p2}
			&\sum_{s=k}^{\infty}\mathbb{E}[V(s+1,x(s+1))-V(s,x(s))]\nonumber\\
			&\hspace{4mm}+\sum_{s=k}^\infty \mathbb{E}[x(s)^TQx(s)+u(s)^TRu(s)]\nonumber\\
			&=\sum_{s=k}^{\infty}\mathbb{E}\{[u(s)+(B(s)^TPB(s)+R)^{-1}B(s)^TPA(s)x(s)]^T\nonumber\\
			&\hspace{4mm}\times[R+B(s)^TPB(s)][u(s)+(B(s)^TPB(s)+R)^{-1}\nonumber\\
			&\hspace{4mm}\times B(s)^TPA(s)x(s)]\}-\mathbb{E}\{x(s)^T[B(s)^TPA(s)]^T\nonumber\\
			&\hspace{4mm}\times[R+B(s)^TPB(s)]^{-1}B(s)^TPA(s)x(s)\}\nonumber\\
			&\hspace{4mm}+\sum_{s=k}^{\infty}\mathbb{E}\{x(s)^T[A(s)^TPA(s)-P+Q]x(s)\}. 
		\end{align}
	When $s \to \infty$, the first term of the left-hand side of equation (\ref{p2}) approaches to $\mathbb{E}[V(k,X(k))]$. Also, note that due to Riccati equation (\ref{q1}), the last two terms of the right-hand side of equation (\ref{p2}) is equal to $0$. It indicates that, the value function reaches the minimum when the optimal controller turns to be
\begin{align}\label{p3}	
		u^*(k)=&-\mathbb{E}[B(k)^TPB(k)+R]^{-1}\mathbb{E}[B(k)^TPA(k)]x(k).\nonumber\\
		&\left.\right. 
\end{align}	
\end{proof}
The optimal controller (\ref{n1}) can be further expressed as
\begin{align}\label{n2}
	u^*(k)=&-\mathbb{E}\{[B+\bar{B}\omega(k)]^TP[B+\bar{B}\omega(k)]+R\}^{-1}\nonumber\\
	&\times\mathbb{E}\{[B+\bar{B}\omega(k)]^TP[A+\bar{A}\omega(k)]\}x(k)\nonumber\\
	=&-\mathbb{E}[B^TPB+B^TP\bar{B}\omega(k)+\omega^T \bar{B}^TPB\nonumber\\
	&+\omega(k)^T\bar{B}^TP\bar{B}\omega(k)]^{-1}\mathbb{E}[B^TPA+B^TP\bar{A}\omega(k)\nonumber\\
	&+\omega(k)^T \bar{B}^TPA+\omega(k)^T\bar{B}^TP\bar{A}\omega(k)]x(k)\nonumber\\
	=&-[B^TPB+\mu B^TP\bar{B}+\mu\bar{B}^TPB\nonumber\\
	&+(\mu^2+\sigma^2)\bar{B}^TP\bar{B}]^{-1}[B^TPA+\mu B^TP\bar{A}\nonumber\\
	&+\mu \bar{B}^TPA+(\mu^2+\sigma^2) \bar{B}^TP\bar{A}]x(k),
\end{align}
which suggests that the optimal controller is closely related to the expectation and covariance of the random parameter. 

Therefore, when the expectation or covariance of the random parameter is unknown, the standard solution (\ref{n2}) is not applicable any more. In order to tackle with this situation, \cite{du} presented a novel algorithm. 

In \cite{du}, the authors have rewritten the generalized Riccati equation into the form of
	\begin{align}\label{p6}
		P=\Pi\left(\mathbb{E}\left[\begin{matrix}
			Q+A(k)^TPA(k)&A(k)^TPB(k)\\
			B(k)^TPA(k)&B(k)^TPB(k)+R
		\end{matrix}\right]\right),
\end{align}
where 
$$
	\Pi(P)=P_{xx}-P_{xu}P_{uu}^{\dag}P_{ux}
$$
is defined as a mapping for a matrix $P$ according to the partition $$P=\left[\begin{matrix}
	P_{xx}&P_{xu}\\
	P_{ux}&P_{uu}
\end{matrix}\right].$$
In this case, the optimal controller can be formulated as 
\begin{align*}
	u^*(k)=&\Gamma\hspace{-1mm}\left(\mathbb{E}\left[\begin{matrix}
		Q+A(k)^TPA(k)&A(k)^TPB(k)\\
		B(k)^TPA(k)&B(k)^TPB(k)+R
	\end{matrix}\right]\right)\\
&\times x(k),
\end{align*}
where 
\begin{align*}
	\Gamma(P)=-P_{uu}^{\dag}P_{ux}.
\end{align*}
Let
\begin{align}\label{p5}
	G=\mathbb{E}\left[\begin{matrix}
		Q+A(k)^TPA(k)&A(k)^TPB(k)\\
		B(k)^TPA(k)&B(k)^TPB(k)+R
	\end{matrix}\right],
\end{align}
the Riccati equation (\ref{p6}) can be equivalently written as 
\begin{align} \label{p4}
	P=\Pi(G). 
\end{align}
Substituting (\ref{p4}) into equation (\ref{p5}) yields
\begin{align}\label{p7}
	G=\mathbb{E}\left[\begin{matrix}
		Q+A(k)^T\Pi(G)A(k)&A(k)^T\Pi(G)B(k)\\
		B(k)^T\Pi(G)A(k)&B(k)^T\Pi(G)B(k)+R
	\end{matrix}\right].
\end{align}
In this scenario, solving the Riccati equation (\ref{p6}) can be converted into seeking for the zero point of equation 
(\ref{p7}). Du-Meng-Zhang \cite{du} applied an iterative stochastic approximation algorithm to solve equation (\ref{p7}):
\begin{align}\label{c11}
	G(k+1)=G(k)+\alpha(k)Y(G(k)),
\end{align}
where 
\begin{align}\label{yg}
&Y(G(k))\nonumber\\
=&\left[\begin{matrix}
		Q+A(k)^T\Pi(G(k))A(k)&A(k)^T\Pi(G(k))B(k)\nonumber\\
	B(k)^T\Pi(G(k))A(k)&B(k)^T\Pi(G(k))B(k)+R
	\end{matrix}\right]\nonumber\\
	&-G(k)
\end{align}		
and $\alpha(k)$ is the learning rate sequence satisfying:
\begin{align*}\begin{matrix}	
		\sum_{k=0}^{\infty}\alpha(k)=\infty& and &\sum_{k=0}^{\infty}\alpha(k)^2\le\infty
	\end{matrix}.
\end{align*}

\begin{lemma}
	Let $\{G(k)\}$ be the sequence constructed by stochastic approximation algorithm (\ref{c11}) for $k=1,2,\dots$. 
	
	Then, the following statements are equivalent:
	\begin{itemize}
		\item[a.] The LQ problem (1)-(2) is well-posed;
		\item[b.] ARE (\ref{p6}) admits a solution $P\textgreater 0$;
		\item[c.] $\{G(k)\}$ is bounded with a positive probability;
		\item[d.] $\{G(k)\}$ converges almostly surely (a.s.) to a deterministic matrix $G^*\in\mathbb{S}_+^{m+n}$.
	\end{itemize}
	Moreover, if either statement is valid, one has the following properties: 
	\begin{itemize}
		\item[(1)]The value function $V(x)=x^TPx$ for all $x\in\mathbb{R}^n$;
		\item[(2)]The solution of ARE (\ref{p6}) is given by $P=\Pi(G^*)$;
		\item[(3)]The optimal control is given by $u^*(k)=\Gamma(G^*)x(k)$;
		\item[(4)] {\small$G^*=\mathbb{E}\left[\begin{matrix}
				Q+A(k)^T\Pi(G^*)A(k)&A(k)^T\Pi(G^*)B(k)\\
				B(k)^T\Pi(G^*)A(k)&B(k)^T\Pi(G^*)B(k)+R
			\end{matrix}\right].$}
	\end{itemize}
\end{lemma}
\begin{remark}
	The algorithm proposed by \cite{du} is centralized. That is, the algorithm requires the entire information of $G(k)$ and other system parameters to update $G(k+1)$ iteratively. However, this is not adaptive to problems where information security and privacy are emphasized. It has become mainstream to replace centralized algorithms with distributed ones where global information is not needed. 
	
\end{remark}

\section{Main Results}
In this section, we present a distributed stochastic approximation method for the Q-learning algorithm for stochastic LQ control. The consensus analysis and convergence analysis of the proposed algorithm   are also provided. 
\subsection{A Distributed Q-Learning Algorithm}
Firstly, we present a novel distributed algorithm. Comparing to the centralized algorithm proposed by \cite{du}, where a single sensor has to collect all information, our algorithm involves N sensors where each sensor only has access to partial information. In particular, an undirected graph $\mathcal{G}=\{\mathcal{V},\mathcal{E}\}$, which contains a vertex set $\mathcal{V}$ and an edge set $\mathcal{E}$, is formed by N sensors. Sensor $j$ is said to be a neighbor of sensor $i$, if $i$ and $j$ are connected by an edge. The set of the neighbors of sensor $i$ is denoted by $\mathcal{N}_i$. Each sensor $i=1,\dots, N$ is assigned to collect measurement data and carry out the estimates of $G^*$. Each sensor is also able to share local estimates with its neighboring sensors via the communication links between them. The graph $\mathcal{G}$ is assumed to be connected throughout the paper. 

Based on the graphical model above, our algorithm is implemented in each sensor concurrently. Specifically, each sensor iteratively computes
\begin{align}
&G_i(k+1)\nonumber\\
=&G_i(k)+\sum_{j\in\mathcal{N}_i}(G_j(k)-G_i(k))+\alpha(k)L_iY(G_i(k))\label{c4},
\end{align}
%Here, {\small\begin{eqnarray}
%		Y(G_i(k))\hspace{-3mm}&=&\hspace{-3mm}\left[\begin{matrix}
%			Q+A(k)^T\Pi(G_i(k))A(k)&\hspace{-1mm}A(k)^T\Pi(G_i(k))B(k)\\
%			\hspace{-6mm}B(k)^T\Pi(G_i(k))A(k)&\hspace{-6mm}B(k)^T\Pi(G_i(k))B(k)+R
%		\end{matrix}\right]\nonumber\\
%		&&\hspace{-2mm}-G_i(k)	\nonumber
%\end{eqnarray}}
where $\sum_{i=1}^{N}L_i=NI$ and $Y(.)$ has been defined in equation (\ref{yg}).
\begin{remark}
	In the proposed algorithm (\ref{c4}), sensor $i$ iteratively computes its estimate for $G^*$ using local information. That is, sensor $i$ only needs to use $G_i(k)$ and $G_j(k), j\in\mathcal{N}_i$ to update $G_i(k+1)$ instead of knowing the estimates of the whole network. 
\end{remark}

\subsection{Boundness of the Distributed Algorithm}
In this part, we first study the boundness of algorithm (\ref{c4}). 
\begin{theorem}
	Under the condition that ARE (\ref{q1}) has a positive-definite solution $P\textgreater 0$, then $\{G_i(k),k\geq 0\}$
	is bounded with a positive probability for each $i=1,2,\cdots,N$.
\end{theorem}

\begin{proof}	
Since $Q>0$ and $R>0$, it follows from (\ref{p7}) that
$$
		G\geq diag\{Q,R\}\geq\varepsilon I,\nonumber
$$
	where $\varepsilon>0$.
	In addition, from the non-decreasing property of $\Pi$ in \cite{du}, it follows from (\ref{p6}) that we have 
$$
		P\geq \Pi(diag\{Q,R\})\geq \varepsilon I.\nonumber
$$
	Thus, there exist unitary matrices $T_1$ and $T_2$ such that
$$
		T_1'PT_1=I, T_2'G_{uu}T_2=I.\nonumber
$$
	Let $T=\left[
	\begin{array}{cc}
		I & 0 \\
		C & I \\
	\end{array}
	\right]\left[
	\begin{array}{cc}
		T_1 & 0 \\
		0 & T_2 \\
	\end{array}
	\right]
	$ with $C=-G_{uu}^{-1}G_{ux}$ and define
	$\tilde{\Upsilon}(k)=T_1^{-1}\Upsilon(k)T,$ $\tilde{N}(k)=T'NT$ where $N=\left[
	\begin{array}{cc}
		Q & 0 \\
		0 & R \\
	\end{array}
	\right]
	$ and $\Upsilon(k)=\left[
	\begin{array}{cc}
		A(k) & B(k) \\
	\end{array}
	\right]$.\\	
	In view of the fact that $\Pi(T'GT)=T_1'\Pi(G)T_1$, we obtain
	\begin{align}
		I=&T_1'PT_1\nonumber\\
		=&T_1'\Pi(G)T_1\nonumber\\
		=&\Pi\Big(\mathbb{E}[\tilde{\Upsilon}'(k)\tilde{\Upsilon}(k)+\tilde{N}]\Big).\label{n16}
	\end{align}
	By letting $\tilde{G} _i(k)=T'G_i(k)T$, it follows from (\ref{c4}) that we have 	
	\begin{align}
	\tilde{G}_i(k+1)=&\tilde{G}_i(k)+\sum_{j\in\mathcal{N}_i}(\tilde{G}_j(k)-\tilde{G}_i(k))\nonumber\\
		&+\alpha(k)\Big(\tilde{\Upsilon}'(k)\Pi(\tilde{G}_i(k))\tilde{\Upsilon}(k)+\tilde{N}-\tilde{G}_i(k)\Big).\label{ns15}
	\end{align}
	We now prove that $\tilde{G}_i(k)$ is bounded a.s. To this end, we denote
	\begin{align}
		\tilde{\Phi}_G(k)=&\tilde{\Upsilon}'(k)\Pi(G(k))\tilde{\Upsilon}(k)+\tilde{N},\nonumber\\
		\tilde{\Phi}(G(k))=&\mathbb{E}\Big(\tilde{\Upsilon}'(k)\Pi(G(k))\tilde{\Upsilon}(k)+\tilde{N}\Big),\nonumber\\
		\tilde{\Psi}_G(k)=&\tilde{\Upsilon}'(k)G_{xx}(k)\tilde{\Upsilon}(k)+\tilde{N},\nonumber\\
		\tilde{\Psi}(G(k))=&\mathbb{E}\Big(\tilde{\Upsilon}'(k)G_{xx}(k)\tilde{\Upsilon}(k)+\tilde{N}\Big).\nonumber
	\end{align}
	It can be shown that $\tilde{\Psi}(G(k))$ is a contraction mapping under matrix $2$-norm. In fact, from (\ref{n16}), it is obvious that $I=\tilde{\Psi}(I)$. Together with $N>0$, it follows that $\tilde{N}>0$. Then there exists a positive number $\lambda<1$ such that
	$\mathbb{E}\Big(\tilde{\Upsilon}'(k)\tilde{\Upsilon}(k)\Big)=I-\tilde{N}\leq \lambda I$.
	Accordingly, for any $M_1$ and $M_2$, it follows that
	\begin{align*}
		&\Vert\tilde{\Psi}(M_1)-\tilde{\Psi}(M_2)\Vert_2\nonumber\\
		=&\Big\Vert\tilde{\Upsilon}'(k)(M_{1,xx}-M_{2,xx})\tilde{\Upsilon}(k)\Big\Vert_2\nonumber\\
		\leq& \Vert M_{1,xx}-M_{2,xx}\Vert_2\Vert\mathbb{E}\Big(\tilde{\Upsilon}'(k)\tilde{\Upsilon}(k)\Big)\Vert_2\nonumber\\
		\leq&\lambda\Vert M_1-M_{2}\Vert_2,\nonumber
	\end{align*}
	this implies that $\tilde{\Psi}(M)$ is a contraction mapping with respective to $M$ by using $\lambda<1$.

	Define:
	\begin{align}
		\mathcal{G}_i(k+1)=&\mathcal{G}_i(k)+\sum_{j\in\mathcal{N}_i}(\mathcal{G}_j(k)-\mathcal{G}_i(k))\nonumber\\
		&+\alpha(k)[\tilde{\Psi}_{\mathcal{G}_i}(k)-\mathcal{G}_i(k)],\label{n17}
	\end{align}
	with initial value $\mathcal{G}_i(0)=T'G_i(0)T$.
	Since ${\mathcal{G}_i}_{xx}(k)\geq \Pi(\mathcal{G}_i(k))$, we have that
$$
		\tilde{\Psi}_{\mathcal{G}_i}(k)\geq \tilde{\Phi}_{\mathcal{G}_i}(k).\label{n18}
$$
	Thus, it holds that
$$
		\tilde{G}_i(k)\leq \mathcal{G}_i(k), k=0,1,2,\cdots,
$$
	which gives that $\mathcal{G}_i(k)$ is an upper bound process of $\tilde{G}_i(k)$.	
	
Let $\hat{G}_i(k)=\mathcal{G}_i(k)-I$, it follows from (\ref{n17}) and $I=\mathbb{E}[\tilde{\Upsilon}'(k)\tilde{\Upsilon}(k)]+\tilde{N}$ that,	
	\begin{align}
	&\hat{G}_i(k+1)\nonumber\\
		=&[1-\alpha(k)][\mathcal{G}_i(k)-I]+\sum_{j\in\mathcal{N}_i}\Big([\mathcal{G}_j(k)-I]-[\mathcal{G}_i(k)-I]\Big)\nonumber\\
		&+\alpha(k)[\tilde{\Psi}(\mathcal{G}_i(k))-I]\nonumber\\
		=&[1-\alpha(k)]\hat{G}_i(k)+\sum_{j\in\mathcal{N}_i}[\hat{G}_j(k)-\hat{G}_i(k)]\nonumber\\
		&+\alpha(k)\Big[\tilde{\Upsilon}'(k){G_i}_{xx}(k)\tilde{\Upsilon}(k)-\mathbb{E}\Big(\tilde{\Upsilon}'(k){G_i}_{xx}(k)\tilde{\Upsilon}(k)\Big)\Big]\nonumber\\
		&+\alpha(k)\Big[\mathbb{E}\Big(\tilde{\Upsilon}'(k){G_i}_{xx}(k)\tilde{\Upsilon}(k)\Big)-\mathbb{E}[\tilde{\Upsilon}'(k)\tilde{\Upsilon}(k)]\Big].\nonumber\\
	\end{align}
	Denote 
	\begin{align}
		\Theta_i(k)=&\mathbb{E}\Big(\tilde{\Upsilon}'(k){G_i}_{xx}(k)\tilde{\Upsilon}(k)\Big)-\mathbb{E}\Big(\tilde{\Upsilon}'(k){G_i}_{xx}(k)\tilde{\Upsilon}(k)\Big)\nonumber\\
		&+\mathbb{E}\Big(\tilde{\Upsilon}'(k){G_i}_{xx}(k)\tilde{\Upsilon}(k)\Big)-\mathbb{E}[\tilde{\Upsilon}'(k)\tilde{\Upsilon}(k)],\nonumber
	\end{align}
	together with $\mathbb{E}\Big(\tilde{\Upsilon}'(k)\tilde{\Upsilon}(k)\Big)\leq \lambda I$
	, we have
	\begin{eqnarray}
		\mathbb{E}[\Theta_i(k)|\mathcal{F}(k-1)]
		&\leq& \Vert{G_i}_{xx}(k)-I\Vert\mathbb{E}\Big(\tilde{\Upsilon}'(k)\tilde{\Upsilon}(k)\Big)\nonumber\\
		&\leq& \lambda \Vert{G_i}_{xx}(k)-I\Vert_2I\nonumber\\
		&\leq &\lambda \Vert\hat{G}_i\Vert_2I.
	\end{eqnarray}
	Moreover, it is easy to verify that 
	\begin{align*}
		\mathbb{E}[\Vert\Theta_i(k)\Vert^2|\mathcal{F}(k-1)]\leq 36\mu+30\mu\Vert\hat{G}_1(k)\Vert^2,
	\end{align*}
	where $\mu$ satisfies that $\mathbb{E}[\Upsilon'(k)\Upsilon(k)]+N\leq \mu$.
	
	By applying similar discussions to Lemma 3.4 in \cite{du}, it yields
	that $\hat{G}_i(k)$ converges to $0$ a.s., which implies that $\mathcal{G}_i(k)$ is bounded a.s.. As a consequence with $\tilde{G}_i(k)=T'G_i(k)T$ and $\tilde{G}_i(k)\leq \mathcal{G}_i(k)$, it follows that
	$G_i(k)$ is bounded a.s..	
\end{proof}	
\subsection{Convergence Analysis}	
We can now prove the convergence of the distributed algorithm (\ref{c4}). The convergence analysis consists of the following two parts: 
\begin{eqnarray}
\hspace{-2mm}&\lim_{k\rightarrow\infty} &\Vert G_i(k)-G_j(k)\Vert=0,\forall i,j \in N_i,\hspace{2mm} a.s.,\label{c1}\\
&\lim_{k\rightarrow\infty}& \Vert G_i(k)-G^*\Vert=0,\forall i \in N_i,\hspace{2mm} a.s.,\label{c2}
\end{eqnarray}
where $G^*$ is the solution to (\ref{p7}).
The first condition (\ref{c1}) indicates that algorithm (\ref{c4}) achieves consensus, and the second condition (\ref{c2}) indicates that the consensus value is the solution to (\ref{p7}).

The detailed proof is given as below.

\subsubsection{Consensus Analysis}
We first focus on proving that each sensor reaches consensus under the distributed scheme (\ref{c4}). 
\begin{theorem}

Suppose that equation (\ref{q1}) has the solution of $P\textgreater 0$, then the proposed algorithm (\ref{c4}) achieves consensus.
\end{theorem}

\begin{proof}

Denote  $F(k)=\mathbf{col}\{G_1(k),G_2(k),\dots,G_N(k)\}$ and $\Phi(k)=\mathbf{col}\{L_1Y_1(k),L_2Y_2(k),\dots,L_NY_N(k)\}$, it follows from (\ref{c4}) that 
\begin{align*}
	F(k+1)=\mathcal{A}F(k)+\alpha(k)\Phi(k),
\end{align*}
where $\mathcal{A}=I-L$ and $L$ is the Laplacian matrix.

Let $M=\frac{1}{N}\textbf{1}_N\textbf{1}_N'$ and $\delta(k)=(I-M)F(k)$, we further obtain
\begin{align*}
	\delta(k+1)=&(I-M)\mathcal{A}F(k)+\alpha(k)(I-M)\Phi(k)\nonumber\\
	=&(\mathcal{A}-M)\delta(k)+\alpha(k)(I-M)\Phi(k),\label{c5}
\end{align*}
where the facts $\mathcal{A}M=M\mathcal{A}=M^2=M$ have been used in the derivation of the last equality.

By applying iterative calculation to (\ref{c5}), it yields that 
\begin{align}
	\delta(k+1)=&(\mathcal{A}-M)^{k+1}\delta(0)+\sum_{\tau=0}^{k}\alpha(\tau)(\mathcal{A}-M)^{k-\tau}\nonumber\\
	&\times(I-M)\Phi(\tau).\nonumber
\end{align}
To prove that the algorithm achieves consensus in the almost sure sense, the key point is to analyze the norm of $\delta(k)$. 
Specifically, we have:
	\begin{align}
		&\Vert\delta(k+1)\Vert\nonumber\\
		=&\Vert(A-M)^k\delta(0)+\sum\limits_{\tau=0}^{k-1}(A-M)^{k-\tau-1}(I-M)\Phi(\tau)\Vert\nonumber\\
		\leq& \Vert(A-M)^k\delta(0)\Vert+\Big\Vert\sum\limits_{\tau=0}^{k-1}(A-M)^{K-\tau-1}(I-M)\Phi(\tau)\Big\Vert\nonumber\\
		\leq& c\rho^{k}\Vert\delta(0)\Vert+\sum\limits_{\tau=0}^{k-1}\Vert\alpha(\tau)(A-M)^{k-\tau-1}(I-M)\Phi(\tau)\Vert\nonumber\\
		\leq& c\rho^k\Vert\delta(0)\Vert+\sum\limits_{\tau=0}^{k-1}\alpha(\tau)\Vert(A-M)^{k-\tau-1}\Vert\Vert I-M\Vert\Vert\Phi(\tau)\Vert.
\end{align}
Based on the fact that, given a connected graph, it yields that 
\begin{align}
	\Vert(\mathcal{A}-M)^k\Vert\leq c\rho^k,
\end{align}
where $c\textgreater 0$ and ${\rho\in(0,1)}$.
It holds that $\Vert I-M\Vert\textless \infty$, $\Vert\delta(0)\Vert\textless \infty$ and $\Vert\Phi(\tau)\Vert\textless \infty$, we have 
\begin{align*}
	\lim\limits_{k\rightarrow \infty}c\rho^k\Vert\delta(0)\Vert\rightarrow 0
\end{align*}
\begin{align*}
	\lim\limits_{k\rightarrow \infty}\sum\limits_{\tau=0}^{k-1}\Vert(A-M)^{k-\tau-1}\Vert\Vert(I-M)\Vert\Vert\Phi(\tau)\Vert\to 0.
\end{align*}
Thus, \begin{align*}
	\lim\limits_{k\rightarrow \infty}\Vert\delta(k)\Vert=0, a.s.
\end{align*}
That is, equation (\ref{c1}) holds.  
\end{proof}
\subsubsection{Convergence Analysis to the Solution of (\ref{c4})}
In this part, we present discussions on the convergence of equation (\ref{c4}), that is, to prove that equation (\ref{c2}) holds. 
\begin{theorem}
Under the assumption that ARE (\ref{q1}) has a solution $P\textgreater 0$, then $G_i(k)$, $i=1,\dots,N$ converge a.s. to $G(k)$.
\end{theorem} 
\begin{proof}	
We define
$\bar{G}(k+1)=\frac{1}{N}\sum_{i=1}^{N}G_i(k+1)$. It is thus obtained from (\ref{c4}) that 
\begin{align*}
	\bar{G}(k+1)=\bar{G}(k)+\frac{\alpha(k)}{N}\sum_{i=1}^{N}L_iY(G_i(k))\nonumber.
	%&=&\bar{G}(k)+\alpha(k)Y(\bar{G}(k))+\alpha(k)\Big[\frac{1}{N}\sum_{i=1}^{N}L_iY(G_i(k))-Y(\bar{G}(k))\Big].\nonumber
\end{align*}
Note that 
\begin{align*}
	&\sum_{i=1}^{N}L_iY_i(G^*)\nonumber\\
	=&\sum_{i=1}^{N}L_i\times\Big(\left[\begin{matrix}
		Q+A(k)^T\Pi(G^*)A(k)&A(k)^T\Pi(G^*)B(k)\\
		\hspace{0mm}B(k)^T\Pi(G^*)A(k)&B(k)^T\Pi(G^*)B(k)+R
	\end{matrix}\right]\nonumber\\
	&-G^*\Big)\nonumber\\
	=&0.\nonumber
\end{align*}
Also, recall the centralized algorithm (\ref{c11}), i.e.,
$$
	G(k+1)=G(k)+\alpha(k)Y(G(k)),\nonumber
$$
and let $\Delta(k)=\bar{G}(k)-G(k)$, the iteration equation of $\Delta(k)$ is given by 
\begin{align}
	&\Delta(k+1)	=\Delta(k)+\frac{\alpha(k)}{N}\sum_{i=1}^{N}L_iY(G_i(k))-\alpha(k)Y(G(k))\nonumber\\
	=&\Delta(k)+\frac{\alpha(k)}{N}\Big[\sum_{i=1}^{N}L_iY(G_i(k))-NY(G(k))\Big]\nonumber\\
	=&\Delta(k)+\alpha(k)\Big(Y(\bar{G}(k))-Y(G(k))\Big)\nonumber\\
	&+\frac{\alpha(k)}{N}\Big[\sum_{i=1}^{N}L_i\Big(Y(G_i(k))-Y(\bar{G}(k))\Big)\Big]\nonumber\\
	=&[1-\alpha(k)]\Delta(k)+\alpha(k)\times\left[\begin{matrix}
		A(k)^TWA(k)&A(k)^TWB(k)\\
		B(k)^TWA(k)&B(k)^TWB(k)
	\end{matrix}\right]\nonumber\\
	&+\frac{\alpha(k)}{N}\Big[\sum_{i=1}^{N}L_i\Big(Y(G_i(k))-Y(\bar{G}(k))\Big)\Big],\label{c12}
\end{align}
where $W=\Pi(\bar{G}(k))-\Pi(G(k)$.
The derivation of (\ref{c12}) depends on the following fact:
\begin{align}
	&Y(\bar{G}(k))-Y(G(k))\nonumber\\
	=&\left[\begin{matrix}
		Q+A(k)^T\Pi(\bar{G}(k))A(k)&A(k)^T\Pi(\bar{G}(k))B(k)\\
		B(k)^T\Pi(\bar{G}(k))A(k)&B(k)^T\Pi(\bar{G}(k))B(k)+R
	\end{matrix}\right]\nonumber\\
	&-\bar{G}(k)\nonumber\\
	&-\left[\begin{matrix}
		Q+A(k)^T\Pi(G(k))A(k)&\hspace{-2mm}A(k)^T\Pi(G(k))B(k)\\
		B(k)^T\Pi(G(k))A(k)&\hspace{-2mm}B(k)^T\Pi(G(k))B(k)+R
	\end{matrix}\right]\nonumber\\
	&+G(k)\nonumber\\
	=&\left[\begin{matrix}
		A(k)^TWA(k)&A(k)^TWB(k)\\
		B(k)^TWA(k)&B(k)^TWB(k)
	\end{matrix}\right]-\Big(\bar{G}(k)-G(k)\Big)\nonumber.
\end{align}
Denote
\begin{align}
	\Psi(k)=&\left[\begin{matrix}
		A(k)^TWA(k)&A(k)^TWB(k)\\
		B(k)^TWA(k)&B(k)^TWB(k)
	\end{matrix}\right]\nonumber\\
	&+\frac{1}{N}\Big[\sum_{i=1}^{N}L_i\Big(Y(G_i(k))-Y(\bar{G}(k))\Big)\Big].\nonumber
\end{align}
Then, equation (\ref{c12}) can be reformulated as 
\begin{align}
	\Delta(k+1)
	=[1-\alpha(k)]\Delta(k)+\alpha(k)\Psi(k).\label{c13}
\end{align}
Together with 
$$
	\Vert G(k)\Vert\textless \infty,
$$
it follows that,
\begin{align*}
	\Big\Vert\left[\begin{matrix}
		A(k)^TWA(k)&A(k)^TWB(k)\\
		B(k)^TW)A(k)&B(k)^TWB(k)
	\end{matrix}\right]
	\Big\Vert\textless \infty.
\end{align*}
According to the above analysis, we have 
$$
	\Vert\Psi(k)\Vert\textless \infty. 
$$	
Applying the facts that $0<1-\alpha(k)<1$, $\Vert\Psi(k)\Vert<\infty$ and $\lim_{k\rightarrow\infty} \alpha(k)=0,$ we obtain 
\begin{align}
	\lim_{k\rightarrow\infty}\Vert\Delta(k+1)\Vert=0.\label{c14}
\end{align}
This gives the second condition (\ref{c2}). 

\end{proof}
Now we are ready to give the main result of this paper. 
\begin{theorem}
Under the assumption that the ARE (\ref{q1}) has a positive-definite solution, the distributed algorithm  (\ref{c4}) is able to converge a.s. to $G^*$.
\end{theorem}
\begin{proof}
On the basis of the analysis given above, we prove that the sensors are able to reach consensus and their consensus states will converge to   
\begin{align*}
	\lim_{k\rightarrow\infty} \Vert G_i(k)-G_j(k)\Vert&=0,\forall i,j \in N_i,\hspace{1mm} a.s.,\\
	\lim_{k\rightarrow\infty} \Vert\bar{G}(k)-G(k)\Vert&=0,\forall i,\hspace{1mm}a.s.,
\end{align*}
which further imply that 
$$
	\lim_{k\rightarrow \infty}\Vert G_i(k)-G(k) \Vert=0, \forall i \hspace{1mm}a.s. 
$$
Based on the convergence analysis in \cite{du}, we finally obtain
$$
	\lim_{k\rightarrow \infty}\Vert G_i(k)-G^* \Vert=0, \forall i,\hspace{1mm}a.s. 
$$
Thus, the distributed stochastic approximation algorithm converges a.s. to $G^*$. 
\end{proof}

%The convergence of (\ref{c14}) depends on the following fact:$\lim_{k\rightarrow\infty}E\VertG_i(k)-\bar{G}(k)\|^2=0$, for and $\varepsilon>0$, there exists $k_N$ such that for $k>k_N$, it holds that
%$E\|G_i(k)-\bar{G}(k)\|^2<\varepsilon$. This implies that for $k>k_N$, there holds that  
%\begin{eqnarray}
%E\|Y(G_i(k))-Y(\bar{G}(k))\|^2<\varepsilon
%\end{eqnarray}
%}
\section{Numerical Example}
We implement the distributed Q-learning algorithm (\ref{c4}) in a discrete-time model. Suppose that the system parameters are given as followed: 
$A=\left[\begin{matrix}
0.2&0\\
0&0.6
\end{matrix}\right]$, $\bar{A}=\left[\begin{matrix}
0.7&0\\
0&0.8\end{matrix}\right]$, $B=\left[\begin{matrix}
0.7\\
0.3
\end{matrix}\right]$,  $\bar{B}=\left[\begin{matrix}
0.1\\
0.7
\end{matrix}\right]$, $Q=\left[\begin{matrix}
0.4&0\\
0&0.7
\end{matrix}\right]$, $R=1$, $\alpha(k)=(\frac{1}{k+2})^{0.6}$. Assume that random parameters follow the Gaussian distribution $N(\mu,\sigma^2)$, where $\mu=1$, $\sigma^2=0.1$. The system is associated with a networked system where sensors $i=1,\dots, 4$ are employed to autonomously calculate $G_i(k)$, which are the estimates of the Q-factor in the $k$-th iteration. The induced graph of the networked system is shown by Fig.~\ref{fig1}. Fig. \ref{fig3} illustrates the $1$-norm of $G_i(k)$ for $i=1,\dots, 4$ after running the proposed algorithm for $200$ times. Fig. \ref{fig2} reveals that for $i=1$, $G_1(k)$ converges to the correct solution $G^*$. Also, other sensors have similar convergence behaviors as sensor $1$ does. 
\begin{figure}[H]
\begin{center}
	\includegraphics[width=4cm]{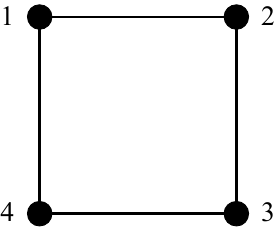}
\end{center}
\caption{A 4-sensor graph}\label{fig1}
\end{figure}

\begin{figure}[H]
\begin{center}
	\includegraphics[width=9cm]{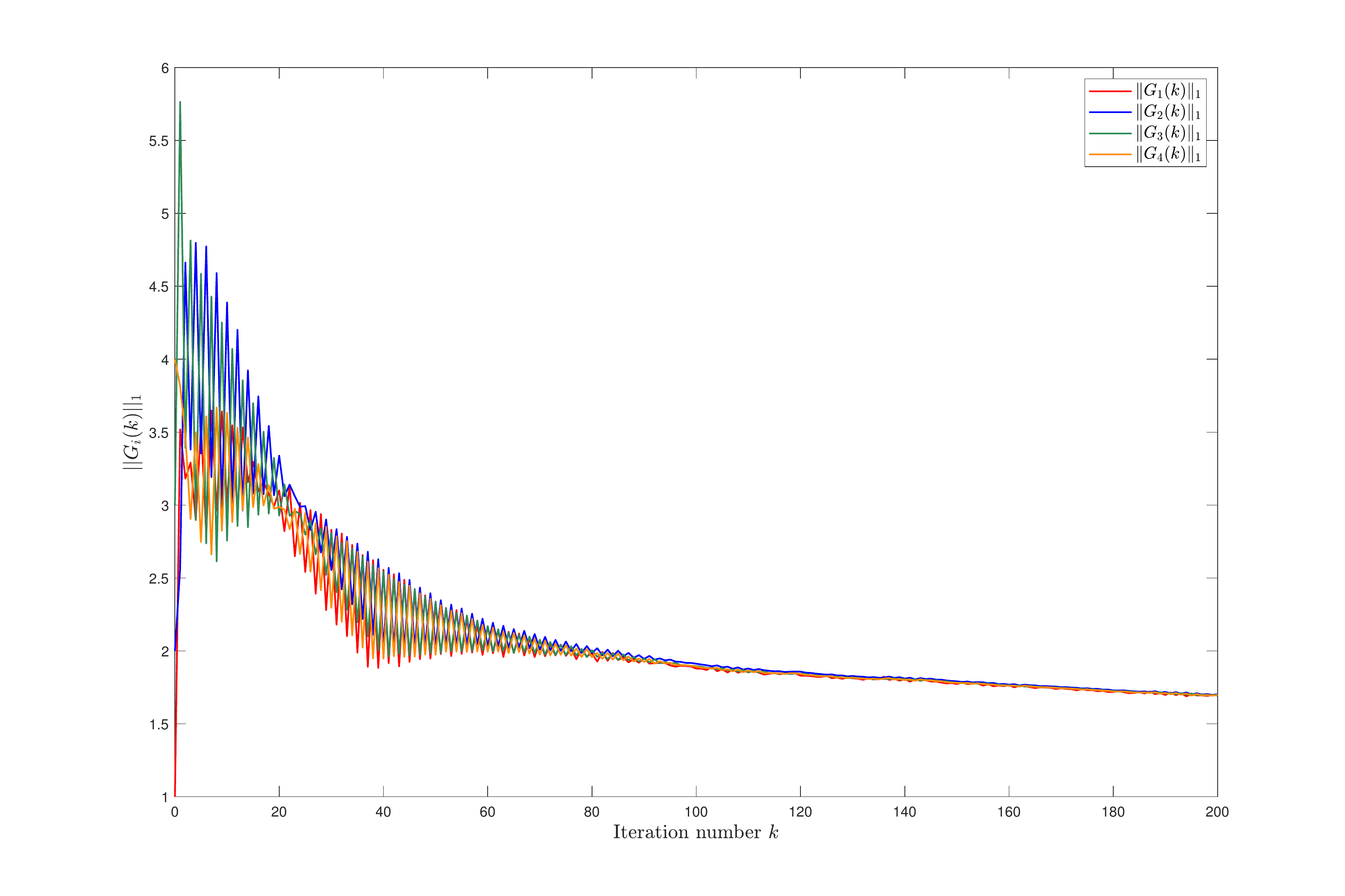}
\end{center}
\caption{Performance of the distributed algorithm}\label{fig3}
\end{figure}

\begin{figure}[H]
\begin{center}
	\includegraphics[width=9cm]{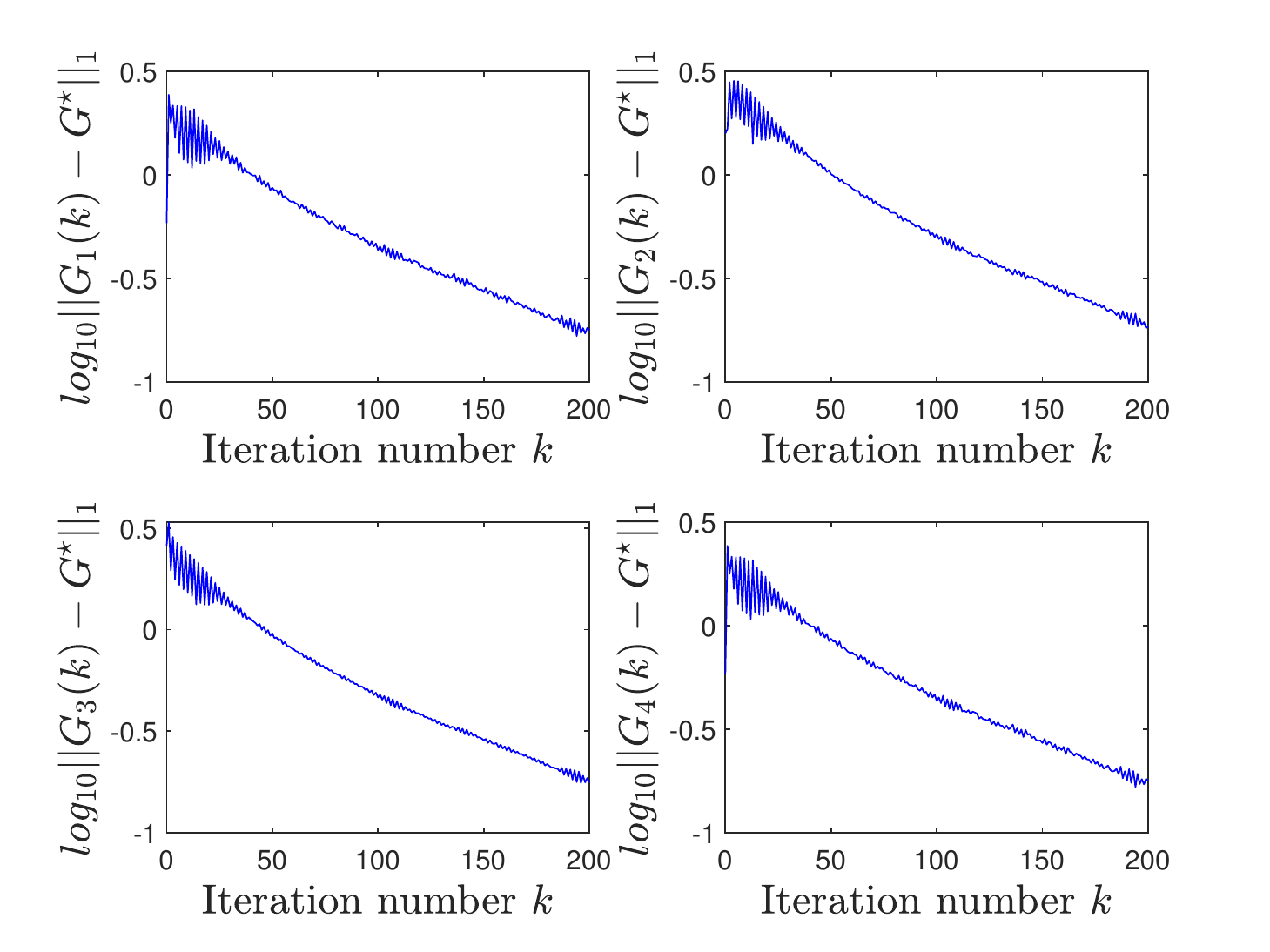}
\end{center}
\caption{Performance of the distributed algorithm}\label{fig2}
\end{figure}

\section{Conclusion}
This paper presents a distributed Q-learning algorithm for stochastic LQ control with random parameters whose statistical properties are unknown. We have proved that the correct solution to the Riccati equation and the optimal controller under the proposed distributed scheme can be derived. Simulation results have verified our analysis. In the future, we are motivated to address stochastic LQ control problems associated with both unknown uncertainties and time delays by using distributed methods.

\end{document}